\documentclass[reqno,12pt]{amsart}
\usepackage[latin1]{inputenc}
\usepackage[english]{babel}
\usepackage{amsmath, amsthm, amssymb, amsopn, amsfonts, amstext, enumerate, color, mathtools, mathrsfs, url, enumitem}

\usepackage{xcolor}

\usepackage[final]{showkeys}

\usepackage{hyperref}

\usepackage[margin=1.15in]{geometry}

\hypersetup{
	colorlinks=true,
	linkcolor=blue,
	citecolor=red,
	urlcolor=black,
	linktoc=all
}

\newcommand{\widex}{\widehat{x}}

\newcommand{\R}{\mathbb{R}}

\newcommand{\Z}{\mathbb{Z}}

\newcommand{\PV}{\mbox{\normalfont P.V.}}
\newcommand{\Haus}{\mathcal{H}}

\newcommand{\delomega}{\partial\Omega}

\def\XXint#1#2#3{{\setbox0=\hbox{$#1{#2#3}{\int}$ }
\vcenter{\hbox{$#2#3$ }}\kern-.6\wd0}}

\newlength{\dhatheight}

\numberwithin{equation}{section}

\theoremstyle{plain}
\newtheorem{definition}{Definition}[section]
\newtheorem{theorem}[definition]{Theorem}

\newtheorem{lemma}[definition]{Lemma}

\theoremstyle{definition}

\renewcommand{\leq}{\leqslant}
\renewcommand{\ge}{\geqslant}
\renewcommand{\geq}{\geqslant}

\title[A fractional Sobolev inequality on convex hypersurfaces]{A fractional Michael-Simon Sobolev inequality\\ on convex hypersurfaces}

\author{Gyula Csat\'o}

\author{Prosenjit Roy}

\address{\vspace{-\baselineskip}
\newline
\textit{Gyula Csat\'o \textsuperscript{1,2}}
\newline
\textsuperscript{1} Universitat de Barcelona,
Facultat de Matem\`atiques i Inform\`atica,
Gran Via de les Corts Catalanes 585,
08007 Barcelona, Spain
\newline
\textsuperscript{2} 
Centre de Recerca Ma\-te\-m\`{a}\-ti\-ca,
Campus de Bellaterra, Edifici C,
08193 Bellaterra, Barcelona, Spain
\newline
\textit{E-mail address}: \textit{\tt gyula.csato@ub.edu}
}

\address{\vspace{-\baselineskip}
\newline
\textit{Prosenjit Roy \textsuperscript{3}}
\newline
\textsuperscript{3} Department of Mathematics and Statistics, IIT-Kanpur, India
\newline
\textit{E-mail address}: \textit{\tt prosenjit@iitk.ac.in}
}

\thanks{The first author is supported by the Spanish grants PID2021-123903NB-I00, PID2021-125021NA-I00, and RED2018-102650-T funded by MCIN/AEI/10.13039/501100011033 and by ERDF ``A way of making Europe'', and by the Catalan grant 
2021-SGR-00087. 
This work is supported by the Spanish State Research Agency, through the Severo Ochoa and Mar\'ia de Maeztu Program for Centers and Units of
Excellence in R\&D (CEX2020-001084-M).
Major part of this work was carried out when the second author was visiting the University of Barcelona (UB). The second author would like to thank UB for their hospitality.  Research work of second author is supported by the Core Research Grant (CRG/2022/007867) of SERB}

\keywords{fractional Sobolev inequalities on manifolds, nonlocal mean curvature, convexity,  fractional mean curvature flow, maximal time of existence.}

\subjclass[2010]{26D10, 46E35, 52A20, 53A07.}

\begin{document}

\title[ Hardy Inequality ]{A fractional Hardy-Sobolev inequality of Michael-Simon type on convex hypersurfaces }

\subjclass[2020]{26D10, 46E35, 52A20, 53A07.}
\keywords{fractional Hardy Inequality, fractional Sobolev inequality, interpolation, hypersurfaces, nonlocal mean curvature, convexity}

\maketitle
\begin{abstract}
In this paper we prove a fractional version of a Caffarelli-Kohn-Nirenberg type interpolation inequality on hypersurfaces $M\subset\R^{n+1}$ which are boundaries of convex sets. The inequality carries a universal constant independent of $M$ and involves the fractional mean curvature of $M.$ In particular, it interpolates between the fractional Micheal-Simon Sobolev inequality recently obtained by Cabr\'e, Cozzi, and the first author, and a new fractional Hardy inequality on $M$. Our method, when restricted to the plane case $M=\R^n$, gives a new simple proof of the fractional Hardy inequality.  To obtain the fractional Hardy inequality on a hypersurface, we establish an inequality which bounds a weighted perimeter of $M$ by  the standard perimeter of $M$ (modulo a universal constant), and which is valid for all convex hypersurfaces $M$.  
\end{abstract}

\section{Introduction}
The classical Michael-Simon and Allard inequality, proven for the first time independently by \cite{Allard} and \cite{Michael-Simon}, extends the classical Sobolev inequality to submanifolds of $\mathbb{R}^{n+1}$ with an universal constant independent of the manifold. In the context of  ~$C^2$ hypersurfaces  $M\subset\R^{n+1}$ it states that there 
exists a constant~$C$ depending only on~$n$ and~$p$, such that
\begin{equation}
 \label{eq:intro:classical MS}
  \|u\|_{L^{p^{\ast}}(M)}\leq C \, \Big( \|\nabla_{\! M} u\|_{L^p(M)}+\|Hu\|_{L^p(M)} \Big) \quad
  \text{for all }u\in C_c^1(M),
\end{equation}
where $1\leq p<n$, $p^{\ast} := {np}/({n-p}),$  $\nabla_{\! M}$ is the tangential gradient on~$M$, and~$H$ is the mean curvature of~$M.$ 

A few years ago a new proof has been found by Brendle \cite{Brendle}, who established the best constant in \eqref{eq:intro:classical MS} in the case ~$p=1$. This solved the long standing open problem on the isoperimetric problem on minimal surfaces (i.e. ~$H=0$). 
We recommend Cabr\'e and Miraglio  \cite{Cabre-Miraglio2022} for a proof of \eqref{eq:intro:classical MS} which combines the original ideas of \cite{Allard} and \cite{Michael-Simon} in an optimal way to give the shortest and easiest proof  currently available in the literature.

We emphasize that the constant~$C$ does not depend on~$M$ and therefore all the information about the geometry of~$M$ is encoded in its mean curvature term ~$H$ that appears on the right-hand side of~\eqref{eq:intro:classical MS}. Henceforth, by universal constant we mean such a constant, which depends only on parameters, but not on the manifold. The parameters in this paper will be some powers $p$ in ~$L^p$ norms and powers of the weight ~$|x|$, i.e., the Euclidean norm of $x\in\R^{n+1}$.

The Hardy inequality also has extensions to submanifolds with a universal constant. A good survey of these extensions, with several generalizations and new proofs can be found in Cabr\'e and Miraglio \cite{Cabre-Miraglio2022}. Among these, of particular interest to the present paper is \cite[Corollary 1.5 and 4.1]{Cabre-Miraglio2022} which shows the following inequality for hypersurfaces $M\subset\R^{n+1}$ with $n\geq 3$:
\begin{equation}
 \label{eq:intro:Cabre-Miraglio}
   \||x|^{\gamma}u\|_{L^{\tau}(M)} \leq C 
   \Big( \|\nabla_{\! M} u\|_{L^p(M)}+\|Hu\|_{L^p(M)} \Big) \quad
  \text{for all }u\in C_c^1(M),
\end{equation}
where $1\leq p<n,$ $\gamma\in [-1,0]$, $\tau\in [p,p^*]$ and 
$$
  \frac{1}{\tau}+\frac{\gamma}{n}=\frac{1}{p^*}\quad\text{ and }\quad
  p^*=\frac{np}{n-p}.
$$
Note that \eqref{eq:intro:Cabre-Miraglio} gives an interpolation between the Sobolev inequality \eqref{eq:intro:classical MS} (taking ~$\tau=p^*$ and ~$\gamma=0$) and Hardy inequality (taking $\tau=p$ and ~$\gamma=-1$) on manifolds.

In the present paper we deal with the nonlocal analogues of the interpolation Hardy-Sobolev inequality \eqref{eq:intro:Cabre-Miraglio}. Actually, we will obtain something more general, namely a Caffarelli-Kohn-Nirenberg type interpolation inequality with universal constant for fractional Sobolev spaces on hypersurfaces of Euclidean space which are the boundary of a convex set. For this purpose we shall define the following  seminorm on a hypersurface ~$M\subset\R^{n+1}$: 
\begin{equation}\label{eq:intro:seminorm}
  [u]_{W^{s,p}(M)}:=\left(\int_M\int_M \frac{|u(x)-u(y)|^p}{|x-y|^{n+sp}}dxdy\right)^{\frac{1}{p}}.
\end{equation}
Here ~$|x-y|$ is the Euclidean distance in the ambient space ~$\R^{n+1}.$ Let us denote by $W^{s,p}(M)$ the space of those functions $u\in L^p(M)$ such that  $[u]_{W^{s,p}(M)}<\infty.$

This seminorm \eqref{eq:intro:seminorm} was first introduced in Cabr\'e and Cozzi \cite{Cabre-Cozzi} where it was used to obtain a gradient estimate for nonlocal minimal graphs. It was precisely in \cite[Proposition 5.2]{Cabre-Cozzi} that for a first time a universal fractional Michael-Simon Sobolev inequality was established, albeit for surfaces satisfying a certain density estimate (which includes local minimal surfaces or nonlocal minimizing minimal surfaces). They proved on such hypersurfaces the inequality
\begin{equation}
 \label{eq:intro:Cabre Cozzi frac Sob ineq}
  \|u\|_{L^{p^{\ast}_s}(M)}\leq C \, [u]_{W^{s,p}(M)}\quad\text{for all } u\in W^{s, p}(M),
\end{equation}
where ~$p_s^*=np/(n-sp)$ is the fractional Sobolev exponent and ~$C$ is a universal constant. The inequality \eqref{eq:intro:Cabre Cozzi frac Sob ineq} holds true on nonlocal minimal surfaces as well (surfaces with vanishing fractional mean curvature), thanks to the recent result of Thompson \cite{Thompson}, who established the required density estimate for such surfaces.

We point out some other reasons why \eqref{eq:intro:seminorm} is a natural definition of the seminorm on a hypersurface of $\R^{n+1}$. The seminorm ~$[\cdot]_{W^{(1+s)/2,2}(\partial E)}$ appears in the second variation formula for the nonlocal  $s$-perimeter of a set $E\subset \R^{n+1}$, see \cite[Appendix B]{LawsonCones} or \cite[Section 6]{FFMMM}, in exact analogy of how the norm ~$\|\nabla_{\partial E} \cdot\|_{L^2(\partial E)}^2$ appears in the second variation formula of the classical perimeter (see also \cite[Section 6]{FFMMM}). Finally, in  \cite{Cabre-Cozzi-Csato}, this definition of the seminorm was precisely the right one to obtain an estimate for the extinction time of the fractional mean curvature flow for convex hypersurfaces. For a further justification of this definition we refer to  \cite[Appendix A]{CabreCsatoMas}, where the corresponding half-Laplacian has been defined on the unit circle in $\R^2$ ---meaning that $[u]_{W^{1/2,2}(M)}^2=\int_M u (-\Delta)^{1/2}u$ .

Our definition \eqref{eq:intro:seminorm} of the seminorm on a manifold does not agree with the ``canonical" definition  in Caselli, Florit, and Serra \cite{caselli2024fractional} where a different definition is given for $p=2$ and for general closed Riemannian manifold. This is due to the fact that $|x-y|$ in our definition of $[\cdot]_{W^{s,p}}$ is the Euclidean distance in $\R^{n+1}$, which for general manifolds is not comparable with the geodesic distance on the manifold. In particular, our corresponding fractional Laplacian on the hypersurface, which is determined by partial integration using our definition \eqref{eq:intro:seminorm}, does not satisfy $(-\Delta)^s\circ(-\Delta)^t=(-\Delta)^{s+t},$ in contrast to the canonical definition given in \cite{caselli2024fractional}. This property cannot be satisfied even if we modify our definition by some multiplicative constant.\footnote{This can be easily seen by taking the hypersurface as the boundary of $(-\infty,\infty)\times(0,1)\in\R^2$, whose boundary consists of two parallel lines $M_1$ and $M_2$. Taking $u$ to be equal to two different constants on each $M_i$ gives that the Laplace-Beltrami operator is $-\Delta_Mu=0$, but $(-\Delta_M)^{1/2}\circ (-\Delta_M)^{1/2} u\neq 0$ provided we take the definition of the fractional Laplacian induced by \eqref{eq:intro:seminorm}. The same example works if $M$ is the union of two concentric circles.}
However, observe that all the motivations mentioned in the previous paragraph come from problems posed in the ambient space $\R^{n+1}$, using that the manifold is embedded in $\R^{n+1}$ and has codimension $1$. This makes our setting different from the application which motivates \cite{caselli2024fractional}: their main application is  Caselli, Florit, and Serra \cite{SerraYausconj} which deals with the existence of nonlocal minimal surfaces on closed Riemannian manifolds.

Before stating our main theorem, we recall the definition of the nonlocal (or fractional) ~$\alpha$-mean curvature ~$H_{\alpha}(x)$  of ~$\Omega\subset\R^{n+1}$ at the point $x \in \partial \Omega$. It is given by 
\begin{equation} \label{Halphadef}
 \begin{split}
H_{\alpha}(x) = H_{\alpha}[\Omega](x) :=& \frac{\alpha}{2} \, \PV \int_{\R^{n+1}}\frac{\chi_{\R^{n + 1} \setminus \Omega}(y)-\chi_{\Omega}(y)}{|y-x|^{n+1+\alpha}} \, dy
\\
=&
   \PV \int_{\delomega} \frac{(y - x) \cdot \nu(y)}{|y - x|^{n + 1 + \alpha}}\,dy,
\end{split}
\end{equation}
where ~$\nu$ denotes the exterior unit normal vector to~$\partial\Omega$.

In order to state our theorem, note first that if~$\Omega\subset\R^{n+1}$ is an open convex set, then~$\partial\Omega$ is a Lipschitz hypersurface and thus, by Rademacher's theorem, differentiable at almost every point~$x\in\partial\Omega.$ On the other hand, by either one of the expressions in \eqref{Halphadef}, we see that~$H_{\alpha}(x)$ is a well-defined quantity in~$[0,+\infty],$ since~$\Omega$ is convex. Furthermore, by Aleksandrov's theorem,~$\partial\Omega$ is (pointwise) twice differentiable at almost every~$x\in\partial\Omega.$ At these points, the nonlocal mean curvature~$H_{\alpha}(x)$ is finite.

In Cabr\'e, Cozzi, and Csat\'o \cite{Cabre-Cozzi-Csato} a fractional analogue of the classical Michael-Simon Sobolev inequality \eqref{eq:intro:classical MS} was proven for convex hypersurfaces.
Here we prove an extension of that result to a Caffarelli-Kohn-Nirenberg type interpolation inequality. Our main theorem is the following, which gives exactly \cite[Theorem 1.3]{Cabre-Cozzi-Csato} for the choice ~$a=1$ and ~$\gamma=0.$ Both results are valid on manifolds which are boundaries of convex sets and  it is an open problem whether it is also true on more general manifolds. 

\begin{theorem}
    \label{th: intro: Csato-Prosenjit}
    Let~$n \ge 1$ be an integer, ~$a\in (0,1]$, $\alpha, 
 s, \in(0,1)$,   ~$\gamma\leq 0$, ~$p,q\geq 1$,  and $\tau >0$ be such that ~$sp <n$,
 \begin{equation}
     \label{eq:intro: relation between the parameters s,p, tau, gamma}
 \frac{1}{\tau} + \frac{\gamma}{n} = \frac{a}{p_s^{*}} + \frac{1-a}{q}, \hspace{3mm} \textrm{where}, \   p_s^{*} = \frac{np}{n-sp},
 \end{equation}
and
\begin{equation}\label{eq:gy:curious condition}
  \gamma\geq -as \quad\text{ if }\quad a=1\text{ or }q=p_s^*.
\end{equation} 
 
 Let~$\Omega\subset\R^{n+1}$ be an open convex set. Then, there exists a constant~$C$ depending only on $a$, $n$, $\alpha$, $s$, $\gamma$, $p$, $q$, and $\tau$ such that
\begin{equation}
     \label{eq: intro: Our Main result}
     \||x|^\gamma u\|_{L^{\tau}(\delomega)}\leq C\left(\frac{1}{2}\int_{\delomega}\int_{\delomega} \frac{|u(x)-u(y)|^p}{|x-y|^{n+ s p}} \, dx dy +\int_{\delomega}H_{\alpha}(x)^{\frac{s p}{\alpha}} |u(x)|^p \, dx  \right)^{\frac{a}{p}} 
      \|u\|_{L^{q}(\delomega)}^{1-a},
 \end{equation}
 for every ~$u\in W^{s,p}(\partial\Omega).$
\end{theorem}

Note that no relation between the parameters~$s$ and~$\alpha$ is assumed within the theorem. Note that we also make no assumption on where the domain ~$\Omega$ is located with respect to the origin, which could be on its boundary ~$\partial\Omega$; the term ~$|x|^{\gamma}$ on the left hand side of the inequality is just the Euclidean distance of ~$x$ in ~$\R^{n+1}$ to the power ~$\gamma.$

We recall that \eqref{eq:intro: relation between the parameters s,p, tau, gamma} is a necessary condition and follows from dimensional balance, since the theorem must hold for ~$\lambda\Omega$ too, for all ~$\lambda>0.$ 

The nonlocal analogue of the classical result of Caffarelli-Kohn-Nirenberg \cite{CKN}, in Euclidean space, has been established  in full generality in Nguyen and Squassina \cite{Squassina2018}. Both of these references establish interpolation inequalities where the product on the right hand side of the inequality also contains weighted norms, weighted by powers of $|x|$. The conditions that \cite{Squassina2018}, respectively its analogues in the local case \cite{CKN}, impose on the parameters coincide  with our conditions, when taking the powers of these weights equal to zero. 
In view of \eqref{eq:intro: relation between the parameters s,p, tau, gamma} and $a\in (0,1]$ the condition \eqref{eq:gy:curious condition} is equivalent to
\begin{equation*}
  \gamma\geq -as \quad\text{ if }\quad \frac{1}{\tau} + \frac{\gamma}{n} = \frac{1}{p_s^{*}},
\end{equation*} 
which is a more adequate formulation if one wishes to compare with  \cite{Squassina2018} or \cite{CKN}.

For an elementary proof of the fractional Hardy inequality in ~$1$-dimension and  references on its earliest proofs we refer to \cite[Section 3]{Hardy1Dim}. The fractional Hardy inequality with optimal constant in terms of the limiting behaviour in  parameter $s$ was obtained in ~\cite{mazya}, i.e., they prove the fractional Hardy inequality such that it contains the local Hardy inequality if taking the limit $s\to 1$.  A major improvement of this last result was given by Frank and Seiringer \cite{FrankSeiringer} who determined exactly the sharp constant in the  fractional Hardy inequality.

Our idea of the proof is based on the observation that the proof of the fractional Sobolev inequality in $\R^n$ presented in \cite{HitchhikersG} (and originally based on the technique developed in Savin and Valdinoci \cite{Savin-Valdinoci1,Savin-Valdinoci2}) actually also gives the fractional Hardy inequality in ~$\R^n$ if combined with the following elementary rearrangement inequality. If ~$E^*$ denotes the Schwarz rearrangement of ~$E\subset\R^{n}$ (i.e., ~$E^*$ is a ball centered at the origin, and with the same measure as ~$E$), then
\begin{equation}
 \label{eq:intro:symm in plane}
  \int_{E}|x|^{-\beta}dx \leq \int_{E^*}|x|^{-\beta}dx= \frac{C(n)}{n-\beta} |E|^{\frac{n-\beta}{n}},
\end{equation} 
where ~$\beta\in[0,n)$ and ~$C(n)$ is a dimensional constant depending only on ~$n.$ Our proof of Theorem \ref{th: intro: Csato-Prosenjit} restricted to the plane case ~($\partial\Omega=\R^n\times\{0\}$) gives a new simple proof of the fractional Hardy inequality in the plane, and is based on precisely this combination of \cite[Section 6]{HitchhikersG} with \eqref{eq:intro:symm in plane}. 

It is clear that \eqref{eq:intro:symm in plane} cannot hold for some universal constant for all ~$E\subset M$  if $M$ can be any hypersurface in ~$\R^{n+1}$. To see this it is sufficient to take a sequence of manifolds ~$M$ which ``oscillate" more and more near the origin (in the sense that ~$M$ tends to have infinitely many leaves which all pass close to the origin, like tending to a single plane with infinite multiplicity, or to a set which is dense in a neighborhood of the origin). Nevertheless, if ~$M=\partial\Omega$ and ~$\Omega$ is convex, then this cannot happen and we can show the following result. In the next theorem the word \emph{measurable} refers to the~$n$-dimensional Hausdorff measure ~$\mathcal{H}^n$ on a hypersurface of ~$\R^{n+1}.$

\begin{theorem}
    \label{th:intro:key theorem}
     Let~$\Omega\subset\R^{n+1}$ be an open convex set and ~$\beta \in [0,n)$. Then, there exists a constant~$C$ depending only on~$n$ and $\beta$  such that for any measurable subset $E$ of $\delomega$,

     \begin{equation}
         \label{eq:intro: main estimate}
\int_{E}|x|^{-\beta}dx \leq C |E|^{\frac{n-\beta}{n}}.
     \end{equation}
\end{theorem}

Our proof of Theorem \ref{th:intro:key theorem} is quite elementary and certainly does not give the best constant ~$C$ in \eqref{eq:intro: main estimate}. We expect that the best constant is achieved by a limiting set ~$\Omega$ which is an infinitely thin circular cylinder with the origin at its center, and whose boundary is thus twice an ~$n$-dimensional plane ball centered at ~$0$. This is supported by many examples we have calculated. In particular, the best constant is certainly not ~$C(n)/(n-\beta)$ appearing in the analogous inequality \eqref{eq:intro:symm in plane} of the plane case, which is sharp. But we were not able to prove \eqref{eq:intro: main estimate} with this conjectured sharp constant. 

The proof of our main Theorem \ref{th: intro: Csato-Prosenjit} for the case $1/\tau\leq a/p+(1-a)/q$ uses, among others, the methods developed in Cabr\'e, Cozzi, and Csat\'o \cite{Cabre-Cozzi-Csato}, Theorem \ref{th:intro:key theorem}, and the new rearrangement based proof of the fractional Hardy inequality. To extend the result to the full range of parameters, i.e., to $1/\tau > a/p+(1-a)/q$ we need some variants and applications of Theorem \ref{th:intro:key theorem}; see Lemma \ref{lemma:weighted integrals on M}. That lemma provides the estimates
$$
   \int_{M\setminus B_1(0)}|x|^{-\beta}dx\leq C(n,\beta)\quad \text{ if }\beta>n\quad\text{ and }
   \quad\int_{M\cap B_1(0)}|x|^{-\beta}dx\leq C(n,\beta)\quad \text{ if }\beta<n,
$$
for some universal constant $C$,
provided $M$ is a hypersurface in $\R^{n+1}$ which is the boundary of a convex set and $B_1(0)$ is the unit ball in $\R^{n+1}.$ These estimates are needed in a certain interpolation argument ---already known in the plane case--- to obtain the full range of parameters appearing in our main theorem. Such estimates cannot hold without the additional convexity assumption, as can be easily seen by similar examples as the ones indicated prior to Theorem \ref{th:intro:key theorem}.

\bigskip

\noindent\textbf{Acknowledgements.} We are grateful to Xavier Cabr\'e who made some helpful comments on an earlier version of this manuscript.

\subsection{Notation}

Throughout the paper, the word \emph{measurable} refers to the~$n$-dimensional Hausdorff measure~$\mathcal{H}^n$ on a hypersurface~$M$ of~$\R^{n+1}$, if not stated explicitly otherwise. The measure of a set~$E \subset M$ will be denoted by~$|E|$ and the integration element simply by~$dx$, instead of $d\Haus^n(x).$ Open balls are understood as balls in the ambient space~$\R^{n+1},$ i.e.,~$B_R(x)=\{y\in \R^{n+1}:\,|y-x|<R\}$ and~$|y-x|$ is the Euclidean distance in~$\R^{n+1}.$ If~$x=0$ we write~$B_R=B_R(0),$ while~$\mathbb{S}^n = \partial B_1$ is the~$n$-dimensional unit sphere in~$\R^{n+1}.$ Moreover, we define ~$\omega_n=|\{x\in\R^n:\,|x|<1\}|,$ the measure of a flat ball in ~$\R^n$, and ~$\alpha_{n-1}$ is the ~$\mathcal{H}^{n-1}$-measure of the boundary of the flat ball in ~$\R^n$ (and thus $\alpha_{n-1}=n\omega_n$).

\section{Proof of Theorem \ref{th:intro:key theorem}}

In this section we will prove Theorem \ref{th:intro:key theorem}. The idea is to write the boundary of ~$\Omega$ as a union of  graphs over the ~$n+1$ canonical hyperplanes passing through the origin, such that the slopes of these graphs is not too big. Essential will be that for convex sets there is a universal bound on the number of such graphs. Since we do not assume that $\Omega$ is bounded, and also for the sake of settling the notation, we first summarize in an elementary lemma what can happen when constructing these graphs by projecting $\Omega$ onto a hyperplane.

\begin{lemma}
  \label{lemma:projection}
   Let  $\Omega\subset\R^{n+1}$ be an open convex set and define $P$ as the projection map $x=(x',x_{n+1})\in\R^{n+1}\mapsto x'$ restricted to $\overline{\Omega}$. If $K_{n+1}=P(\overline{\Omega})$ is the image, we denote for $x'\in K_{n+1}$
   $$
     g^{-}(x'):=\inf\{(P^{-1}(x'))_{n+1})\},\qquad g^+(x'):=\sup\{(P^{-1}(x'))_{n+1}\},
   $$
    where ~$(P^{-1}(x'))_{n+1}$ is the last coordinate of the preimage of ~$x'$.
   
   Then ~$g^-(x')$ (respectively ~$g^+(x')$) is either finite for all ~$x'\in K_{n+1}$ or ~$g^-(x')=-\infty$ (respectively ~$g^+(x')=+\infty$) for all ~$x'\in K_{n+1}.$

   Moreover,  
   $$
     \Omega=\{(x',x_{n+1}):\, x'\in K_{n+1}\text{ and }g^-(x')<x_{n+1}<g^+(x')\}
    $$
    and the function ~$g^-$ is convex (respectively ~$g^+$ is concave) in the convex set ~$K_{n+1}$.
\end{lemma}

\begin{proof}
    Assume by contradiction that there exists two points ~$x',y'\in K_{n+1}$ such that ~$g^-(x')$ is finite but ~$g^-(y')=-\infty$. Let ~$x=(x',x_{n+1})$ and ~$(y',y_{n+1})$ be in ~$\Omega.$ But then by convexity the segment ~$[(x',x_{n+1}),(y',r)]\subset\Omega$ for all ~$r\leq y_{n+1}.$  By letting ~$r\to -\infty$ we obtain that every point ~$(x',t)$ is in $\overline{\Omega}$ if ~$t<g^-(x')$, which is a contradiction to the definition and finiteness of ~$g^-(x')$, since ~$\overline{\Omega}$ is convex. 

    The second statement of the lemma is obvious, since by convexity of $\Omega$ we have that ~$(x',x_{n+1})\in \Omega$ if ~$g^-(x')<x_{n+1}<g^+(x')$. 

    Let us first prove the statement of the lemma claiming that ~$g^-$ is convex in the case that ~$g^-$ is finite and ~$g^+=+\infty$. In that case ~$\lambda (x',s)+(1-\lambda)(y',t)\in\Omega$ for every ~$s>g^-(x')$ and ~$t>g^-(y')$ and every ~$0<\lambda<1$. Thus by definition of ~$g^-$ we have that ~$g^-(\lambda x'+(1-\lambda)y')\leq \lambda s+(1-\lambda)t$. Letting ~$s$ go to ~$g^-(x')$, and ~$t$ go to ~$g^-(y')$ respectively, gives the desired inequality. The other cases are proven in a similar way. In the case that ~$g^-(x')=g^+(x')$ for some ~$(x',x_{n+1})=(x',g^{\pm}(x'))$, then one only needs to introduce the parameter ~$t$, but $s$ has to be taken to equal to $x_{n+1}$ in the previous argument.  
\end{proof}

In the next Lemma we prove that ~$\partial\Omega$ is the union (up to rotations) of not more than ~$2(n+1)$ graphs whose slopes are uniformly bounded by little more than ~$\sqrt{n}.$ For the next lemma to make sense, recall that a Lipschitz function is differentiable almost everywhere.

\begin{lemma}
 \label{gy:lemma:split into Mi}
  Let ~$\epsilon>0$ and ~$\Omega\subset\R^{n+1}$ be a convex open set. Then there exists  at most ~$2(n+1)$ closed sets ~$W_i\subset\R^n$ with nonempty interior, ~$i=1,\ldots,2(n+1)$, and Lipschitz functions ~$g_i:W_i\to \R$, such that the graphs of $g_i$
  $$
    \{(z,g_i(z)):\, z\in W_i\}=:M_i\,,
  $$
  satisfy that ~$\partial\Omega$ is the union 
  $$
    \partial\Omega=\bigcup_{i=1}^{2(n+1)}M_i\qquad\text{modulo rotations of each $M_i$} 
  $$
  up to a set of zero measure,
  and
  $$
    |\nabla g_i|\leq \sqrt{n+\epsilon}\quad\text{ in }W_i\,.
$$
\end{lemma}

Note that the last statement on the upper bound on ~$|\nabla g_i|$ makes sense in the closed set ~$W_i$ (and not only in the interior), since a Lipschitz function can be extended to whole ~$\R^n$ with the same Lipschitz constant; see for instance \cite[Theorem 16.16]{CsatoDacKneuss}. Hence ~$\nabla g_i$ exists almost everywhere in ~$W_i$. 

Each function  $g_i$ is actually either convex or concave, but we will not need this statement.

We have introduced ~$\epsilon>0$ to make the proof as short and easy as possible, which certainly makes, in general,  many ~$M_i$ overlap. We have put no effort in constructing disjoint ~$M_i$, since even in that case  the method of proof would not give what we believe is the best constant in Theorem \ref{th:intro:key theorem}, and this would just have added unnecessary technical difficulties.

\begin{proof}
\textit{Step 1} Recall that on the boundary ~$\partial\Omega$ of a convex set the unit outward normal vector ~$\nu=(\nu_1,\ldots,\nu_{n+1})$ exists almost everywhere. We define 
$$
  \Gamma_1^-:=\left\{x\in\partial\Omega:\,\nu(x)\text{ exists and }\nu_1(x)< -\frac{1}{\sqrt{1+n+\epsilon}}\right\}.
$$
We assume that ~$\Gamma_1\neq \emptyset$ (if this is not the case, we can jump to the next step) and denote $\widehat{x}_1:=(x_2,\ldots,x_{n+1})\in \R^n$. Thus, using the projection map $P$ of Lemma \ref{lemma:projection} onto the plane $\{x_1=0\}$ (instead of the plane $\{x_{n+1}=0\}$),  the function ~$g_1^-(\widehat{x}_1):=\inf\{(P^{-1}(\widehat{x}_1))_1\}$ is convex (~$g_1^-$ must be finite, since the tangent plane at ~$x$ separates ~$\Omega$ and its complement) and we must have that
$$
  \nu(g_1^-(\widehat{x}_1),\widehat{x}_1)=\frac{(-1,\nabla g_1^-(\widex_1)   )}{\sqrt{1+|\nabla g_1^-(\widex_1)|^2}}
$$
for all ~$\widex_1\in D_1^-$, where ~$D_1^-$ is the set where ~$g_1^-$ is differentiable. We now define 
$$
  W_1^-:=\overline{\{\widex_1\in D_1^-:\,|\nabla g_1^-(\widex_1)|< \sqrt{n+\epsilon}\}\}}.
$$
By definition ~$W_1^-$ is closed. Moreover ~$W_1^-$ has nonempty interior, since for every ~$\widex_1\in W_1^-$ there exists a ~$\delta>0$ such that ~$|\nabla g_1^-(y)|\leq \sqrt{n+\epsilon}\}$ for all ~$|y-\widex_1|<\delta$ and  $y\in D_1^-$  (see \cite[Theorem 25.5, page 246]{Rockafellar}). 
Finally we define
$$
  M_1^-:=\{(g_1^-(\widex_1),\widex_1):\, \widehat{x}_1\in W_1^-\}
$$
Note that by construction ~$\Gamma_1^-\subset M_1^-$.

\smallskip

\textit{Step 2 (Complement of Step 1).} We now define the set (which we omit if it is empty) 
$$
  \Gamma_1^+:=\left\{x\in\partial\Omega:\,\nu(x)\text{ exists and }\nu_1(x)> \frac{1}{\sqrt{1+n+\epsilon}}\right\}.
$$
Proceeding as in Step 1 we define, for the concave function ~$g_1^+:=\sup\{(P^{-1}(\widehat{x}_1))_1\}$, the sets
$$
  W_1^+:=\overline{\{\widex_1\in D_1^+:\,|\nabla g_1^+(\widex_1)|< \sqrt{n+\epsilon}\}\}},
$$
where ~$D_1^+$ is the set where ~$g_1^+$ is differentiable. 
Note that now the first component of $\nu(g_1^+(\widex_1),\widex_1)$ is equal to $(1+|\nabla g_1^+(\widex_1)|^2)^{-1/2}$.
Finally we define 
$$
  M_1^+:=\{(g_1^+(\widex_1),\widex_1):\, \widehat{x}_1\in W_1^+\}.
$$
By construction the set of points $x\in\partial\Omega$ where $\nu$ exists and ~$|\nu_1|>(\sqrt{1+n+\epsilon})^{-1}$ is contained in ~$M_1^-\cup M_1^+$.
\smallskip

\textit{Step 3 (Iteration of Steps 1 and 2 in other coordinate directions).}
By repeating the previous construction in all the coordinate directions we obtain ~$2n+2$ sets ~$\Gamma_i^{\pm}$ such that 
$$
  \nu_i(x)<-\frac{1}{\sqrt{1+n+\epsilon}}\text{ on }\Gamma_i^{-},\text{ and }
  \nu_i(x)>\frac{1}{\sqrt{1+n+\epsilon}}\text{ on }\Gamma_i^{+},
$$
for ~$i=1,\ldots,n+1$. Moreover by construction each of these sets ~$\Gamma_i^{\pm}$ is contained in a graph of a function (up to a rotation) whose gradient is bounded by ~$\sqrt{n+\epsilon}$. 

Now observe that every point ~$x\in \delomega$ where ~$\nu(x)$ exists must be in one of the ~$\Gamma_i^{\pm}.$ If this were not true for some ~$x,$ then 
$$
  |\nu|^2(x)=\sum_{i=1}^{n+1}|\nu_i(x)|^2\leq\frac{n+1}{1+n+\epsilon}<1,
$$
which is a contradiction to ~$|\nu|=1$. Hence, up to a set of measure zero, ~$\delomega$ is contained in the union of the ~$\Gamma_i^{\pm}$ which in turn are contained in the union of the ~$M_i^{\pm}.$ This completes the proof of the lemma. 
\end{proof}

We now  give the proof of Theorem \ref{th:intro:key theorem}

\begin{proof}[Proof of Theorem \ref{th:intro:key theorem}]
Let ~$\epsilon$, ~$W_i$, ~$g_i$ and ~$M_i$ be as in Lemma \ref{gy:lemma:split into Mi}. We can chose $\epsilon=1$ for instance, but any $\epsilon>0$ does the job. For a measurable set ~$E\subset \partial\Omega$ we define ~$E_i:=E\cap M_i$. We shall also abbreviate the parametrizations ~$z\in W_i\mapsto(z,g_i(z))=\varphi_i(z)\in M_i$ and set ~$F_i=\varphi^{-1}(E_i)$. With these abbreviations and using Lemma ~\ref{gy:lemma:split into Mi} we have 
\begin{align*}
    \int_{E_i}\frac{1}{|x|^{\beta}}dx
    =&
    \int_{F_i}\frac{\sqrt{1+|\nabla g_i(z)|^2}}{(|z|^2+g_i^2(z))^{\beta/2}}dz
    \leq
    \sqrt{1+n+\epsilon}\int_{F_i}\frac{1}{|z|^{\beta}}dz
    \leq
    \sqrt{1+n+\epsilon}\int_{F_i^*}\frac{1}{|z|^{\beta}}dz,
\end{align*}
where ~$F_i^*$ is the Schwarz rearrangement of ~$F_i$ in $\R^n$, i.e., ~$F_i^*=\{z\in\R^n:\,|z|<R\}$ such that $\omega_n R^n=|F_i|$. Hence we obtain (recall the notation $\alpha_{n-1}=n\omega_n$)
\begin{align*}
        \int_{E_i}\frac{1}{|x|^{\beta}}dx
    \leq & \alpha_{n-1}\sqrt{1+n+\epsilon}\int_0^{R}r^{n-1-\beta}dr
    =\frac{\alpha_{n-1}\sqrt{1+n+\epsilon}}{n-\beta}\left(\frac{|F_i|}{\omega_n}\right)^{\frac{n-\beta}{n}}.
\end{align*}
Observe that ~$|F_i|\leq |E_i|$, which follows from integrating the inequality ~$1\leq \sqrt{1+|\nabla g_i|^2}$ over ~$F_i$, and thus we have proven that 
\begin{equation}
 \label{eq:gy:estimate for Ei}
   \int_{E_i}\frac{1}{|x|^{\beta}}dx\leq C_1 |E_i|^{\frac{n-\beta}{n}},\quad
   \text{ for }C_1=\frac{n\sqrt{1+n+\epsilon}}{n-\beta} \omega_n^{\frac{\beta}{n}}.
\end{equation}

Theorem \ref{th:intro:key theorem}  now follows by summing up the estimates of \eqref{eq:gy:estimate for Ei} over ~$i=1,\ldots, 2(n+1)$. We obtain from ~\eqref{eq:gy:estimate for Ei} and the generous estimate ~$|E_i|\leq |E|$ that 
\begin{align*}
    \int_{E}\frac{1}{|x|^{\beta}}dx\leq & \sum_{i=1}^{2(n+1)}\int_{E_i}\frac{1}{|x|^{\beta}}dx\leq C_1\sum_{i=1}^{2(n+1)}|E_i|^{\frac{n-\beta}{n}}
    \leq 
    2(n+1)C_1 |E|^{\frac{n-\beta}{n}},
\end{align*}
which proves the theorem with 
~$
  C=2(n+1)C_1
$.
\end{proof}


\section{Proof of the main Theorem for $1/\tau\leq a/p+(1-a)q$}
\label{section:proof of tau bigger p}

In this section we prove the main Theorem \ref{th: intro: Csato-Prosenjit} under the case
\begin{equation}
 \label{eq:assumption case 1}
  \frac{1}{\tau}\leq \frac{a}{p}+\frac{1-a}{q}.
\end{equation}
The proof of the theorem in the complementing case is given in the next section.

For the proof of Theorem \ref{th: intro: Csato-Prosenjit}
we have to present the notations and one of the central lemmas of \cite{Cabre-Cozzi-Csato}.
Let~${u}:{\delomega}\to \R$ be a bounded and non-negative measurable function with compact support. For~$i \in \Z$, we write
\begin{align*}
  A_i & :=\{{u}> 2^i\},\quad a_i:=|A_i|, \  D_i := A_i\setminus A_{i+1}=\left\{ 2^i< {u}\leq 2^{i+1}\right\}, \ \text{ and }\quad d_i:=|D_i|.
\end{align*}
We have that the sets $D_i$ are pairwise disjoint,
$$
  \{ u = 0\} \cup \bigcup_{\substack{ j\in \mathbb{Z}\\j\leq i}}D_j= \partial \Omega \setminus A_{i+1},\quad
  \bigcup_{\substack{ j\in\mathbb{Z}\\ j\geq i}} D_j=A_i\,,\quad\text{ and }\quad 
  a_i=\sum_{\substack{j\in \mathbb{Z}\\ j\geq i}} d_j.
$$

We will also need the following auxiliary lemma---see~\cite[Lemma 6.2]{HitchhikersG},
 which  only uses H\"older's inequality.
Note that, as~$u$ is non-negative, bounded and has compact support, our sequence~$\{ a_i \}$ satisfies the hypotheses of the next lemma for some~$N\in\mathbb{Z}.$

\begin{lemma}
\label{lemma:ineq of sums ak and akplus1}
Let~$s\in (0,1),$~$p\geq 1$ such that~$s p<n$, and~$N\in\mathbb{Z}.$ Suppose~$\{ a_i \}_{i \in \Z}$ is a bounded, non-negative, and non-increasing sequence with 
$$
  a_i=0\quad\text{ for all~$i\geq N$}.
$$
Then
\begin{equation}\label{eq: sec3: inequality involving series}
     \sum_{i\in \mathbb{Z}} 2^{pi} a_i^{(n-s p)/n} \leq 2^{p^{\ast}_s}\sum_{\substack {i\in\mathbb{Z} \\ a_i\neq 0}}2^{pi}a_i^{-s p/n}a_{i+1}.
\end{equation}
\end{lemma}

Note that, from the hypotheses made on the sequence $\{a_{i}\}$ in the lemma, clearly both series are convergent. Moreover after re-indexing the sum in the 
right hand side of \eqref{eq: sec3: inequality involving series},
with ~$i$ replaced by ~$i-1$, \eqref{eq: sec3: inequality involving series} reads as 
\begin{equation}\label{eq: sec3: inequality involving series 2}
     \sum_{i\in \mathbb{Z}} 2^{pi} a_i^{(n-s p)/n} \leq 2^{p^{\ast}_s-p}\sum_{\substack {i\in\mathbb{Z} \\ a_{i-1}\neq 0}}2^{pi}a_{i-1}^{-s p/n}a_{i}.
\end{equation}

The next lemma of \cite{Cabre-Cozzi-Csato} is what we need for  the proof of our main  Theorem \ref{th: intro: Csato-Prosenjit} and it is through this lemma that the fractional mean curvature enters. It is the generalisation of the analogous result \cite[Lemma 6.3]{HitchhikersG} in the flat case to convex hypersurfaces.

\begin{lemma}\cite[Lemma 4.3]{Cabre-Cozzi-Csato}]
\label{Analogue Lemma 6.3 in HG} Let~$s\in (0,1),$~$p \ge 1$ such that~$n > s p$, and~$\Omega\subset\R^{n+1}$ be an open convex set. Let~${u}\in L^{\infty}({\delomega})$ be a non-negative function with compact support. 
Then, 
\begin{align*}
 \frac{1}{2}\int_{{\delomega}}\int_{{\delomega}}\frac{|{u}(x)-{u}(y)|^p}{|x-y|^{n+s p}}\, {dx}{dy}
 +\int_{{\delomega}} H_{\alpha}(x)^{\frac{s p}{\alpha}} {u}(x)^p \, {dx} \ge c \sum_{\substack{i\in\Z \\ a_{i-1}\neq 0}}2^{pi}a_{i-1}^{-s p/n}a_i \,,
\end{align*}
for some constant~$c>0$ depending only on~$n$,~$\alpha$,~$s$, and~$p$.
\end{lemma}
Applying ~\eqref{eq: sec3: inequality involving series 2} the conclusion of the above lemma can be restated as 
\begin{align}\label{eq: sec3: inequality to be used 2}
 \frac{1}{2}\int_{{\delomega}}\int_{{\delomega}}\frac{|{u}(x)-{u}(y)|^p}{|x-y|^{n+s p}}\, {dx}{dy}
 +\int_{{\delomega}} H_{\alpha}(x)^{\frac{s p}{\alpha}} {u}(x)^p \, {dx} \ge c \sum_{i\in\Z }2^{pi}a_{i}^{(n-sp)/n}.
\end{align}
In fact, ~\eqref{eq: sec3: inequality to be used 2}  is the inequality that we will use in the proof of ~Theorem \ref{th: intro: Csato-Prosenjit}. 

\begin{proof}[Proof of Theorem \ref{th: intro: Csato-Prosenjit}. for $1/\tau\leq a/p+(1-a)/q$.]  First, notice that it sufficient to prove the theorem for ~$|u|$, where ~$u\in W^{s,p}(\partial\Omega).$
This follows from the observation that the  left hand side of \eqref{eq: intro: Our Main result} is the same for both ~$u$ and ~$|u|$, whereas in the right hand side one can use the inequality
$$
  ||u(x)|-|u(y)|| \leq |u(x)-u(y)|.
$$
Furthermore, we can assume that $u$ is bounded, using truncation. We can also assume that $u$ has compact support. This last assumption on the support is less obvious, but can be done by a cutoff function identically as in \cite[Section 4, Proof of Theorem 1.3]{Cabre-Cozzi-Csato}. 

 Consider now the left hand side of \eqref{eq: intro: Our Main result}. Using the decomposition,

$$
\bigcup_{\substack{ j\in \mathbb{Z}}}D_j= \partial \Omega\cap\{u>0\}, 
$$
one has 
\begin{equation}
    \label{eq:section 3: decompositon of boundary}
    \int_{\partial \Omega} |x|^{\tau\gamma}|u(x)|^{\tau}dx 
    = \sum_{j \in\mathbb{Z} } \int_{D_j}|x|^{\tau\gamma}|u(x)|^{\tau}dx
    \leq  \sum_{j \in\mathbb{Z} } 2^{(j+1)\tau}\int_{D_j}|x|^{\tau\gamma}dx.
\end{equation}
By assumption we have that $1/\tau+\gamma/n>0$ and thus ~$-\tau\gamma \in [0,n)$. Now  apply ~Theorem \ref{th:intro:key theorem} with ~$\beta = -\tau\gamma$ and ~$E=D_j$ to \eqref{eq:section 3: decompositon of boundary}, to obtain a constant ~$C$ depending only on ~$\tau, \gamma$ and ~$n$, such that
\begin{align*}
    \label{section 3: decompositon of boundary and lemma applied}
    \int_{\partial \Omega} |x|^{\tau\gamma}|u(x)|^{\tau}dx  
    \leq & C \sum_{j \in\mathbb{Z} } 2^{(j+1)\tau}|D_j|^{1+\frac{\tau\gamma}{n}} = (C2^\tau) \sum_{j \in\mathbb{Z} } 2^{j\tau}|D_j|^{1+\frac{\tau\gamma}{n}} \\
    = &(C2^\tau) \sum_{j \in\mathbb{Z} } 2^{j\tau}d_j^{1+\frac{\tau\gamma}{n}} .
\end{align*}
Now use \eqref{eq:intro: relation between the parameters s,p, tau, gamma} to obtain 
\begin{equation}
    \label{eq: sec3: estimates}
      \int_{\partial \Omega} |x|^{\tau\gamma}|u(x)|^{\tau}dx 
      \leq  C \sum_{j \in\mathbb{Z} } 2^{j\tau}d_j^{\frac{a\tau}{p_s^{*}} + \frac{(1-a)\tau}{q}}
      = C \sum_{j \in\mathbb{Z} } \left\{ \left(2^{jp}d_j^{\frac{p}{p_s^*}}\right)^{\frac{\tau a}{p}} \left(2^{jq}d_j\right)^{\frac{\tau (1-a)}{q}} \right\}.
\end{equation} 

We shall now use the following generalized H\"older inequality: 
let ~$t_1, t_2 \geq 0$ with ~$t_1 + t_2 \geq 1$, then for ~$x_k, y_k \geq 0$, one has 
    \begin{equation}
    \label{eq:gy:Holder general}
    \sum_{k\in \mathbb{Z}}x_k^{t_1}y_k^{t_2} \leq \left( \sum_{k\in \mathbb{Z}} x_k\right)^{t_1}\left( \sum_{k\in \mathbb{Z}}y_k \right)^{t_2}.
    \end{equation}
If one exponent is zero, for instance $t_2=0$, then the inequality is understood as $\sum x_k^{t_1}\leq (\sum x_k)^{t_1}$. The proof can be found for instance in \cite[Theorems 11, 22, and 167]{Hardy} after excluding some special elementary cases.

Now applying \eqref{eq:gy:Holder general} with ~$x_k = 2^{jp}d_j^{\frac{p}{p_s^*}}$, ~~$y_k =2^{jq}d_j$, ~$t_1 = \frac{\tau a}{p}$ and ~$t_2= \frac{(1-a)\tau}{q}$ together with \eqref{eq: sec3: estimates} we get, using also \eqref{eq:assumption case 1},
\begin{equation}
    \label{eq: sec3: after apllying the inequality}
      \int_{\partial \Omega} |x|^{\tau\gamma}|u(x)|^{\tau}dx 
      \leq  C \left(\sum_{j\in \mathbb{Z}} 2^{jp}d_j^{\frac{p}{p_s^*}}\right)^{\frac{\tau a}{p}} \left(\sum_{j\in \mathbb{Z}}2^{jq}d_j\right)^{\frac{\tau (1-a)}{q}}.
\end{equation}
Independently, we also have 
\begin{equation}
    \label{eq: sec 3: estimate of lp norm}
    \int_{\partial \Omega} |u(x)|^qdx  = \sum_{j \in\mathbb{Z} } \int_{D_j}|u(x)|^{q}dx \geq  \sum_{j \in\mathbb{Z}} 2^{jq} |D_j| = \sum_{j \in\mathbb{Z}} 2^{jq} d_j.
\end{equation}
Clearly, the result  now follows after applying ~\eqref{eq: sec 3: estimate of lp norm} and ~\eqref{eq: sec3: inequality to be used 2} to ~\eqref{eq: sec3: after apllying the inequality} ---bearing in mind that $d_i\leq a_i$ and $p/p_s^*=(n-sp)/n$.

\end{proof}

\section{Proof of the main Theorem for $1/\tau>a/p+(1-a)/q$ }

In this section we prove the main Theorem \ref{th: intro: Csato-Prosenjit} under the assumption
\begin{equation}
 \label{eq:assumption case 2}
  \frac{1}{\tau}> \frac{a}{p}+\frac{1-a}{q}.
\end{equation}
This assumption, together with the dimensional balance condition \eqref{eq:intro: relation between the parameters s,p, tau, gamma}, gives that
\begin{equation}
    \label{eq:sec4: gamma negative}
    \gamma < -as. 
\end{equation}

The method of the proof follows the strategy used in \cite[Section (V)]{CKN} or \cite{Squassina2018}. This method consists in estimating each of 
$$
  \||x|^{\gamma}u\|_{L^{\tau}(\partial\Omega\cap B_1(0))}\quad\text{ and }\quad
  \||x|^{\gamma}u\|_{L^{\tau}(\partial\Omega\setminus B_1(0))}
$$
separately, by using appropriately H\"older inequality in two different ways to reduce to the case already proven in Section \ref{section:proof of tau bigger p}. However, to carry out this strategy on manifolds, we need some estimates on the integrals of $|x|^{-\beta}$ on the manifold intersected with the unit ball or its complement. Such estimates are only possible under additional assumptions on the manifold and once more we use the convexity assumption.

\begin{lemma}
    \label{lemma:weighted integrals on M}
     Let~$\Omega\subset\R^{n+1}$ be an open convex set, ~$B_1(0)$ is the $n+1$ dimensional ball centered at the origin, and $\beta\in\R$. Then, there exists a  constant~$C_1$ depending only on~$n$ and ~$\beta$  such that 

     \begin{equation}
         \label{eq:ball complement}
\int_{\delomega\setminus B_1(0)}|x|^{-\beta}dx \leq C_1\quad\text{ if }\beta>n.
     \end{equation}
And there exists a  constant ~$C_2$ depending only on~$n$ and ~$\beta$  such that
          \begin{equation}
         \label{eq:bound in ball}
\int_{\delomega\cap B_1(0)}|x|^{-\beta}dx \leq C_2\quad\text{ if }\beta<n.
     \end{equation}
\end{lemma}

\begin{proof}
[Proof of \eqref{eq:ball complement}:]   In view of Lemma \ref{gy:lemma:split into Mi} it is sufficient to prove \eqref{eq:ball complement} with $\partial \Omega$ replaced by $M$,  where $M$ is the graph of a Lipschitz function $g$ whose gradient is bounded, i.e., 
$$M := \left\{(z,g(z)) 
 | \ g : W \subset \mathbb{R}^n \rightarrow \mathbb{R}, \ |\nabla g(z)| \leq  \sqrt{n+1} \right\}. 
 $$
Define the map ~$\phi(z):= (z,g(z))$. In addition, we can assume without loss of generality that $W=\R^n.$  Thus we obtain that
\begin{equation}
    \label{eq: integral representation}
    \int_{M\setminus B_1(0)} \frac{1}{|x|^\beta}dx = \int_{\mathbb{R}^n \setminus \phi^{-1}(B_1(0))} \sqrt{1+|\nabla g(z)|^2} \left( |z|^2+ |g(z)|^2\right)^{-\frac{\beta}{2}}dz.
\end{equation}
For each ~$\sigma \in \mathbb{S}^{n-1}$, define
$$
  r_0(\sigma) := \left\{
   \begin{array}{l}
   \sup \left\{ t\in [0,\infty) \ | \  (s\sigma, g(s\sigma))   
  \in B_1(0)\, \ \forall s \in [0,t] \right\}
   \smallskip \\
   0 \quad\text{ if } \left\{ t\in [0,\infty) \ | \  (s\sigma, g(s\sigma))   
  \in B_1(0)\, \ \forall s \in [0,t] \right\}=\emptyset.
  \end{array}
  \right.
$$
Geometrically, $r_0(\sigma)$ is the first time when the graph of $g$ leaves the unit ball, when moving from the origin in the direction $\sigma.$ 
By definition of $r_0$ we have that 
\begin{equation}
 \label{r0 sigma less one}
  r_0(\sigma)\leq 1\quad\text{ for all }\sigma\in \mathbb{S}^{n-1}.
\end{equation}
By continuity of $g$ we must also have that  
\begin{equation}
    \label{eq: rsigma relation}
    r_0^2(\sigma) + g^2(r_0(\sigma)\sigma) \geq 1.
\end{equation}
We make a change of variables in the integral on the right hand side of \eqref{eq: integral representation} and use the inclusion
$$M\setminus B_1(0) \subset  \left\{ (r\sigma, g(r\sigma)) \ :  \ \sigma \in \mathbb{S}^{n-1}, \ r \geq r_0(\sigma)\right\}$$
to conclude the inequality
\begin{equation*}
    \label{eq: integral representation 2}
    \int_{M\setminus B_1(0)} \frac{1}{|x|^\beta}dx \leq \int_{\mathbb{S}^{n-1}}d\sigma   \int_{r_0(\sigma)}^\infty \sqrt{1+|\nabla g(r\sigma)|^2} \left( |r|^2+ |g(r\sigma)|^2\right)^{-\frac{\beta}{2}}r^{n-1}dr.
\end{equation*}
Using that $|\nabla g|^2\leq n+1$ gives 
\begin{equation*}
    \label{eq: integral representation 3}
    \int_{M\setminus B_1(0)} \frac{1}{|x|^\beta}dx \leq \sqrt{n+2}\int_{\mathbb{S}^{n-1}}d\sigma   \int_{r_0(\sigma)}^\infty  \left( |r|^2+ |g(r\sigma)|^2\right)^{-\frac{\beta}{2}}r^{n-1}dr.
\end{equation*}
Now, using the change of variable ~ $r = s +r_0(\sigma)$ and the estimate \eqref{r0 sigma less one},  we obtain
\begin{equation}
    \label{eq: final inequality to prove}
    \int_{M\setminus B_1(0)} \frac{dx}{|x|^\beta} \leq \sqrt{n+2}\int_{\mathbb{S}^{n-1}}d\sigma   \int_{0}^\infty  \psi(s,\sigma)^{-\frac{\beta}{2}} (1+s)^{n-1}ds,
\end{equation}
where 
$$
\psi(s,\sigma) := \left( |r_0(\sigma) +s|^2+ |g((s+r_0(\sigma))\sigma)|^2\right).
$$

Now our aim is to estimate the function $\psi$ from above. An application of the mean value theorem gives that for some ~$\rho \in (r_0(\sigma), s+r_0(\sigma))$,
$$
g((s+r_0(\sigma))\sigma) = s \left\{\nabla g(\rho\sigma)\cdot \sigma \right\}+g(r_0(\sigma)\sigma).
$$

We have omitted the dependence of $\rho$ on $s$ and $r_0(\sigma)$ as we will obtain a bound on $\psi(s,\sigma)$ which is independent of $\rho $ and $\sigma$, as we will see shortly.

Let ~$\epsilon$ belong to $(0,1)$ which will be chosen appropriately later. Taking the square of the last identity for $g(s+r_0(\sigma))$, and applying the inequality ~$2|ab| \leq  \epsilon a^2 + \frac{1}{\epsilon}b^2$  with ~$a=s \left\{\nabla g(\rho\sigma)\cdot \sigma \right\} $ and ~$b=g(r_0(\sigma)\sigma) $, we obtain
$$
  g^2((s+r_0(\sigma))\sigma) \geq (1-\epsilon)g^2(r_0(\sigma)\sigma) +\left(1-\frac{1}{\epsilon}\right)|\nabla g(\rho\sigma)\cdot \sigma|^2s^2.
  $$
Using Cauchy-Schwarz inequality, ~$|\nabla g |^2 \leq n+1 $, and ~$1-\frac{1}{\epsilon} <0$, we have
\begin{equation}
  \label{eq: final  estimate on g}
     g^2((s+r_0(\sigma))\sigma) \geq (1-\epsilon)g^2(r_0(\sigma)\sigma) + \left(1-\frac{1}{\epsilon}\right)(1+n)s^2.
\end{equation}
Therefore, using ~\eqref{eq: rsigma relation} in the last step of the following estimate, together with ~\eqref{eq: final  estimate on g}, we get
\begin{equation}
\label{eq: total estimate}
 \begin{split}
   \psi(s,\sigma)&=|r_0(\sigma) +s|^2+ |g((s+r_0(\sigma))\sigma)|^2
   \\ 
   &\geq r_0^2(\sigma) +s^2+  (1-     
   \epsilon)g^2(r_0(\sigma)\sigma) +\left(1-\frac{1}{\epsilon}\right)(1+n)s^2 
    \\
     &\geq  (1-\epsilon ) \left\{r_0^2(\sigma) +   g^2(r_0(\sigma)\sigma)\right\} + \left\{1 + 
    \left(1-\frac{1}{\epsilon}\right)(1+n)\right\}s^2
    \\
    &\geq  (1-\epsilon) + \left\{1 + 
    \left(1-\frac{1}{\epsilon}\right)(1+n)\right\}s^2.
   \end{split}
\end{equation}
Choose now an ~$\epsilon = \epsilon_0 \in \left( \frac{n+1}{n+2}, 1\right)$ depending only on $n$. Then clearly both ~$A(n) := 1-\epsilon_0$ and ~$B(n):= 1 + 
 \left(1-\frac{1}{\epsilon_0}\right)(1+n)$ are strictly positive numbers. Using \eqref{eq: total estimate} we have 
$$
  \psi(s,\sigma)^{-\frac{\beta}{2}} \leq \left(A(n) + B(n)s^2 \right)^{-\frac{\beta}{2}} \hspace{3mm} \forall \sigma \in \mathbb{S}^{n-1}.
$$
Therefore as ~$\beta > n ,$
\begin{multline*}
\int_0^\infty \psi(s,\sigma)^{-\frac{\beta}{2}} (1+s)^{n-1} ds \leq \int_0^\infty \left(A(n) + B(n)s^2 \right)^{-\frac{\beta}{2}}(1+s)^{n-1}ds =:C_0(\beta,n) < \infty.
\end{multline*}
Combining this inequality with \eqref{eq: final inequality to prove}
finishes the proof.
\smallskip

\textit{Proof of ~\eqref{eq:bound in ball}.} 
We recall first a classical result concerning the monotonicity of the perimeter of convex sets with respect to inclusion: if ~$A$ and ~$B$ are convex subsets of ~$\R^{n+1}$ such that ~$A\subset B$, then ~$|\partial A|\leq|\partial B|$ (A possible proof is by using Cauchy's surface area formula, see for instance ~\cite[Proposition A.2 in the Appendix]{Cabre-Cozzi-Csato}).

We split the proof into two different cases: $\beta\in[0,n)$ and $\beta<0.$ Let us first prove the case $\beta\in [0,n).$ 
It follows from Theorem ~\ref{th:intro:key theorem}, the inclusions ~$\partial\Omega\cap B_1(0)\subset \partial(\Omega\cap B_1(0))$ and $\Omega\cap B_1(0)\subset B_1(0)$ that
\begin{align*}
  \int_{\delomega\cap B_1(0)}|x|^{-\beta}dx
  \leq &
  \int_{\partial(\Omega \cap B_1(0))}|x|^{-\beta}dx
  \leq  C(n,\beta) |\partial(\Omega \cap B_1(0))|^{\frac{n-\beta}{n}}
  \\
  \leq &
  C(n,\beta)|\partial B_1(0)|^{\frac{n-\beta}{n}}.
\end{align*} 

We now prove the case  ~$\beta < 0$. Using the inequality ~$|x|\leq 1$, one has , after using the monotonicity of the perimeters  for convex sets as in the previous case,
$$\int_{\delomega\cap B_1(0)}|x|^{-\beta}dx
  \leq |\partial B_1(0)|.$$
\end{proof}

In the proof of Theorem \ref{th: intro: Csato-Prosenjit} we will use twice H\"older inequality in the following form, which we present here to avoid repetition. Let $~\tau, ~\gamma, ~\alpha, ~\beta \in \mathbb{R}, ~\alpha  > \tau>0,$ and suppose $E \subset \R^{n+1}$ is measurable. Then 
\begin{equation}
    \label{lemma: sec 4: Holder type}
    \int_{E}|x|^{\tau\gamma}|u(x)|^\tau dx \leq \left( \int_E |x|^{\beta\alpha}|u(x)|^\alpha dx\right)^{\frac{\tau}{\alpha}} \left( \int_E|x|^{\frac{\tau\alpha(\gamma - \beta)}{\alpha -\tau}}dx\right)^{\frac{\alpha -\tau}{\alpha}}
\end{equation}
for all measurable functions $u:E\to\R$.  
The proof of the above lemma is a direct application of H\"older inequality with the exponents ~$\alpha/\tau$ and ~$\alpha/(\alpha-\tau)$.

\begin{proof}[Proof of Theorem \ref{th: intro: Csato-Prosenjit} for ~$1/\tau>a/p+(1-a)/q$]
Recall that $\gamma<-as$ by \eqref{eq:sec4: gamma negative}. Thus the assumption \eqref{eq:gy:curious condition} yields that
$$
    a\neq 1\quad\text{ and }\quad q\neq p_s^*. 
$$ 
By assumption the parameters satisfy the identity
\begin{equation}
    \label{th: end part of the main theorem}
     \frac{1}{\tau} + \frac{\gamma}{n} = \frac{a}{p_s^{*}} + \frac{1-a}{q} =: A(a).
\end{equation}
Observe that
\begin{equation}
  \label{eq:A decreasing}
   \begin{split}
              A\text{ is increasing in $a$ }\quad &\text{ if $q>p_s*$},
             \\ 
             A\text{ is decreasing in $a$ }\quad &\text{ if $q<p_s*$}.
    \end{split}
\end{equation} 
 
Since the inequality of Theorem \ref{th: intro: Csato-Prosenjit} that we wish to prove is scaling invariant under the substitution $\Omega\mapsto\lambda \Omega$ and $u(x)\mapsto cu(x/\lambda)$ we can assume without loss of generality that 
$$\int_{\delomega}\int_{\delomega} \frac{|u(x)-u(y)|^p}{|x-y|^{n+ s p}} \, dx dy +\int_{\delomega}H_{\alpha}(x)^{\frac{s p}{\alpha}} |u(x)|^p \, dx =1$$
and 
$$\|u\|_{L^{q}(\delomega)}=1,$$
since these two norms scale differently.  In view of this  it is sufficient to prove that 
$$ \||x|^\gamma u\|_{L^{\tau}(\delomega)}\leq C$$
for some constant depending only on the parameters and the dimension, but not on $u$ nor $\Omega$.
 Since ~$0<a<1$ we can choose ~$a_1$ and  ~$a_2$  such that
\begin{displaymath}
   \begin{split}
      0<a_1<a<a_2<1\quad \text{ if $q >p_s^*$},       
      \\
      0<a_2<a<a_1<1 \quad \text{ if $q < p_s^*$}.
\end{split}
\end{displaymath}
The values of $a_1$ and $a_2$ will depend on a further condition, which in turn will depend only on $a,$ $n,$ $p,$ $q,$ $s,$ $\gamma,$ and $\tau$.
In view of \eqref{eq:A decreasing} this choice yields that
$$
  A(a_1)< A(a)<A(a_2).
$$
We now choose $\tau_1>0$ and $\gamma_1 $ such that 
$$
  \frac{1}{\tau_1}=\frac{a_1}{p_s^*}+\frac{1-a_1}{q} \quad\textrm{and}\quad   \gamma_1 = 0.
$$
Also choose $\tau_2>0$ and $\gamma_2 < 0$ such that 
$$
  \frac{1}{\tau_2}=\frac{a_2}{p}+\frac{1-a_2}{q} \quad \textrm{and} \quad \gamma_2 = -a_2s.
$$
With this choice one has that the dimensional balance condition \eqref{eq:intro: relation between the parameters s,p, tau, gamma}
$$ \frac{1}{\tau_i} + \frac{\gamma_i}{n} = A(a_i), \ i=1,2 $$
is satisfied for the new set of parameters. Using the identity $1/p_s^*+s/n=1/p$ and the dimensional balance condition leads to the two identities 
\begin{equation}
 \label{eq:taui bigger tau}
   \begin{split}
     \frac{1}{\tau}-\frac{1}{\tau_i}= &
     (a-a_i)\left(\frac{1}{p_s^*}-\frac{1}{q}\right)+\frac{\gamma_i-\gamma}{n}
     \\
     =& (a-a_i)\left(\frac{1}{p}-\frac{1}{q}\right)+\frac{\gamma_i-\gamma-(a-a_i)s}{n},
   \end{split}
\end{equation}
for $i=1,2$. Recall again that $\gamma<-as<0$ by  \eqref{eq:sec4: gamma negative}. 
Thus by the definitions of $\gamma_i$, using the first identity of \eqref{eq:taui bigger tau} for $i=1$, and the second identity for $i=2$ respectively, gives that for $a_1$ and $a_2$ close enough to $a,$ 
\begin{equation}
  \label{eq:taui bigger tau final}
    \tau_i >\tau\quad\text{ for }i=1,2. 
\end{equation}

Since  ~$\frac{1}{\tau_i}\leq \frac{a_i}{p}+\frac{1-a_i}{q}$ for ~$i=1,2$,  we can use  Theorem \ref{th: intro: Csato-Prosenjit}, which we have proven in Section \ref{section:proof of tau bigger p} under these assumptions, to obtain 
\begin{equation*}
     \label{eq: sec4: Our Main result used}
     \||x|^{\gamma_i} u\|_{L^{\tau_i}(\delomega)}\leq C\left(\frac{1}{2}\int_{\delomega}\int_{\delomega} \frac{|u(x)-u(y)|^p}{|x-y|^{n+ s p}} \, dx dy +\int_{\delomega}H_{\alpha}(x)^{\frac{s p}{\alpha}} |u(x)|^p \, dx  \right)^{\frac{a_i}{p}} 
      \|u\|_{L^{q}(\delomega)}^{1-a_i}.
 \end{equation*}
 In view of the assumptions made above on $u$ and $\Omega$, this yields that 
 \begin{equation}
     \label{eq: sec4: after using the known case}
      \||x|^{\gamma_i} u\|_{L^{\tau_i}(\delomega)}\leq C, \quad\text{ for }  i=1,2.
 \end{equation}
 Applying  H\"older inequality \eqref{lemma: sec 4: Holder type} with ~$E={ B_1(0) \cap \partial\Omega}, ~\tau, ~\gamma,
 ~\alpha = \tau_1$ and ~$\beta = \gamma_1$, \eqref{eq:taui bigger tau final}, together with \eqref{eq: sec4: after using the known case},
 we get 
 \begin{equation}
     \label{eq: proof of the theorem last part}
     \begin{split} 
     \int_{ B_1(0) \cap \partial\Omega}|x|^{\tau\gamma}|u(x)|^\tau dx  
     \leq &     
     \left( \int_{ B_1(0) \cap \partial\Omega} |x|^{\gamma_1\tau_1}|u(x)|^{\tau_1} dx\right)^{\frac{\tau}{\tau_1}} \left( \int_{ B_1(0) \cap \partial\Omega}|x|^{\frac{\tau\tau_1(\gamma - \gamma_1)}{\tau_1 -\tau}}dx\right)^{\frac{\tau_1 -\tau}{\tau_1}}
      \\
     \leq & C\left( \int_{ B_1(0) \cap \partial\Omega}|x|^{\frac{\tau\tau_1(\gamma - \gamma_1)}{\tau_1 -\tau}}dx\right)^{\frac{\tau_1 -\tau}{\tau_1}}.
     \end{split}
 \end{equation}
 In view of the dimensional balance condition, the inequality ~$A(a_1)<A(a)$~ is equivalent to  
 $$
 \frac{\tau\tau_1(\gamma - \gamma_1)}{\tau_1 -\tau} >-n.
 $$
Hence we can appeal to  ~\eqref{eq:bound in ball} of Lemma \ref{lemma:weighted integrals on M} to get a constant ~$C >0$ that depends only on $a,n,p,q,s,\tau,$ and $ \gamma, $ such that 
 $$\int_{ B_1(0) \cap \partial\Omega}|x|^{\frac{\tau\tau_1(\gamma - \gamma_1)}{\tau_1 -\tau}}dx \leq C.$$
Thus we obtain from \eqref{eq: proof of the theorem last part} that
\begin{equation}
    \label{eq: proof of the theorem last part1} 
     \int_{ B_1(0) \cap \partial\Omega}|x|^{\tau\gamma}|u(x)|^\tau dx \leq C.
\end{equation}

Applying a second time H\"older inequality \eqref{lemma: sec 4: Holder type} with ~$E={ B_1(0)^c \cap \partial\Omega}, ~\tau, ~\gamma,
 ~\alpha = \tau_2$ and ~$\beta = \gamma_2$, \eqref{eq:taui bigger tau final}, together with \eqref{eq: sec4: after using the known case},
 we get 
 \begin{align*}
     \label{eq: proof of the theorem last part} 
     \int_{\partial\Omega\setminus B_1(0)}|x|^{\tau\gamma}|u(x)|^\tau dx
     \leq  &
     \left( \int_{\partial\Omega\setminus B_1(0)} |x|^{\gamma_2\tau_2}|u(x)|^{\tau_2}dx\right)^{\frac{\tau}{\tau_2}} 
     \left( \int_{\partial\Omega\setminus B_1(0)}|x|^{\frac{\tau\tau_2(\gamma - \gamma_2)}{\tau_2 -\tau}}dx\right)^{\frac{\tau_2 -\tau}{\tau_2}}
      \\
     \leq  &C\left( \int_{\partial\Omega\setminus B_1(0)}|x|^{\frac{\tau\tau_2(\gamma - \gamma_2)}{\tau_2 -\tau}}dx\right)^{\frac{\tau_2 -\tau}{\tau_2}}.
 \end{align*}

 Like before ~$A(a) < A(a_2)$ is equivalent to 
 $$
   \frac{\tau\tau_2(\gamma - \gamma_2)}{\tau_2 -\tau} < -n,
 $$
 and hence we can appeal to  \eqref{eq:ball complement} of Lemma \ref{lemma:weighted integrals on M} to get a constant ~$C >0$ that depends only on ~$a,n,p,q,s,\tau,$ and $\gamma,$ such that 
 $$\int_{\partial\Omega\setminus B_1(0)}|x|^{\frac{\tau\tau_1(\gamma - \gamma_1)}{\tau_1 -\tau}}dx \leq C$$
This shows that
\begin{equation}
    \label{eq: proof of the theorem last part2} 
     \int_{\partial\Omega\setminus B_1(0)}|x|^{\tau\gamma}|u(x)|^\tau dx \leq C.
\end{equation}
Finally combining ~\eqref{eq: proof of the theorem last part1} and  ~\eqref{eq: proof of the theorem last part2} finishes the proof of the theorem.
\end{proof}

\bigskip

\noindent
\textbf{Conflict of interest} On behalf of all authors, the corresponding author states that there is no conflict of interest.

\end{document}